\documentclass{article}
\usepackage{authblk}
\usepackage{hyperref}
\usepackage{amsmath, amssymb, amsfonts, amsthm}
\usepackage{graphicx}
\usepackage[utf8]{inputenc}
\usepackage{float}
\usepackage{xcolor}
\usepackage{tikz}
\usetikzlibrary{calc}
\usetikzlibrary{arrows}
\usepackage[dvipsnames]{xcolor}

\theoremstyle{plain} % estilo de teoremas e similares

\theoremstyle{definition} % estilo de definições e similares

\theoremstyle{remark} % estilo de observações e similares
\newtheorem*{observacao}{Remark}
\newtheorem{theorem}{Theorem}

\newtheorem{property}{Property}

\theoremstyle{definition}
\newtheorem{definition}{Definition}

\theoremstyle{remark}

\title{Lights Out!\\
A game of combinatorics and linear algebra}

{\author[1]{Tiane Marcarini \thanks{tiane.pinto@ufes.br}} \author[2]{Cássio Vieira Morais \thanks{cassio.morais@ufes.br}}}

\affil[1,2]{Departamento de Matemática  

\smallskip

Universidade Federal do Espírito Santo}

\date{\today}

\begin{document}

\maketitle
%
%\begin{abstract}
%Neste trabalho, estudamos uma variante triangular do jogo ``Lights Out", proposto na Olimpíada Capixaba de Matemática de 2025. Apresentamos uma descrição combinatória do jogo, caracterizamos formalmente suas operações e introduzimos o conceito de núcleo, que determina quais configurações admitem solução e quantas soluções existem. Em seguida, analisamos padrões geométricos que surgem em elementos do núcleo e descrevemos mecanismos de propagação desses padrões para jogos de tamanhos maiores. Por fim, modelamos o problema por meio de sistemas
%lineares sobre o corpo $\mathbb{Z}_2$, obtendo uma matriz associada ao jogo e um critério combinatório para a invertibilidade dessa matriz. Esse critério estabelece que o jogo admite solução para toda configuração se, e somente se, o número de recobrimentos do tabuleiro por peças $1\times 1$ e $2\times 1$ é ímpar.
%\end{abstract}

\begin{abstract}
	In this work, we study a triangular variant of the ``Lights Out" game, proposed in the 2025 Capixaba Mathematics Olympiad. We present a combinatorial description of the game, formally characterize its operations, and introduce the notion of a quiet pattern, which determines which configurations admit a solution and how many solutions they possess. We then analyze the geometry of quiet patterns and describe the propagation mechanisms that generate patterns for larger board sizes. Finally, we model the problem using linear systems over the field $\mathbb{Z}_2$, obtaining a matrix associated with the game and a combinatorial criterion for its invertibility. This criterion shows that the game admits a solution for every configuration if and only if the number of coverings of the triangular board by $1\times 1$ and $2\times 1$ tiles is odd.
\end{abstract}

\section{Introduction}
\label{sec:introducao}

While drafting the problems for the 2025 Espírito Santo Mathematics Olympiad (OCMAT) (\cite{ocmat:site}), Tiane, a co-author of this article, presented an idea that would later become Question 2 of the third phase of Level 3. The problem featured a triangular arrangement of circles in two colors (see Figure \ref{fig:antes}), which could be interacted with by swapping the colors of specific circles. Although the exam explored only the basic aspects of these swapping rules, several other questions naturally emerged as we analyzed them more carefully.

Upon discussing these questions as a team, we quickly realized that the problem possessed a very rich structure. Later, we discovered that similar problems had already been proposed in other contexts. One example is the \emph{Troca-Cor} (Color-Swap) question from OBMEP 2009, as well as the game \emph{Lights Out}, which has analogous rules but is formulated on a square board  $5 \times 5$ (see~\cite{jaap:site, wikipedia-lightout, obmep:site}). Although less known in Brazil, this game was once marketed in the United States as a small handheld console. From a mathematical perspective, generalizing this type of game to arbitrary graphs is a research topic in combinatorics (see~\cite{evolution:2025, nearlycomplete:2025} for recent results). In this sense, the version used in OCMAT, based on a triangular arrangement, can be seen as a variation of the original Lights Out or, equivalently, as a particular case of the general problem.

The game can be stated as follows. Consider a triangular arrangement of circles (buttons) of side length $n$ ; we say that this arrangement defines a game of {size} $n$ . Figures~\ref{fig:antes} and \ref{fig:depois} illustrate cases where $n = 7$ . Each button can be in one of two states: lit/on or unlit/off, represented by the colors white and black, respectively. A {configuration} is the choice of a state for each of the buttons.

We will call the {neighbors} of a button all the circles that are tangent to it. The fundamental action of the game consists of {pressing} a button. When this occurs, the pressed button and all its neighbors swap states: those that were lit become unlit, and those that were unlit turn on. Figures~\ref{fig:antes} and \ref{fig:depois} show the configurations before and after {button~10} is pressed.

\begin{figure}[H]
\centering
\begin{minipage}{.45\textwidth}
  \centering
  \includegraphics[width=.85\linewidth]{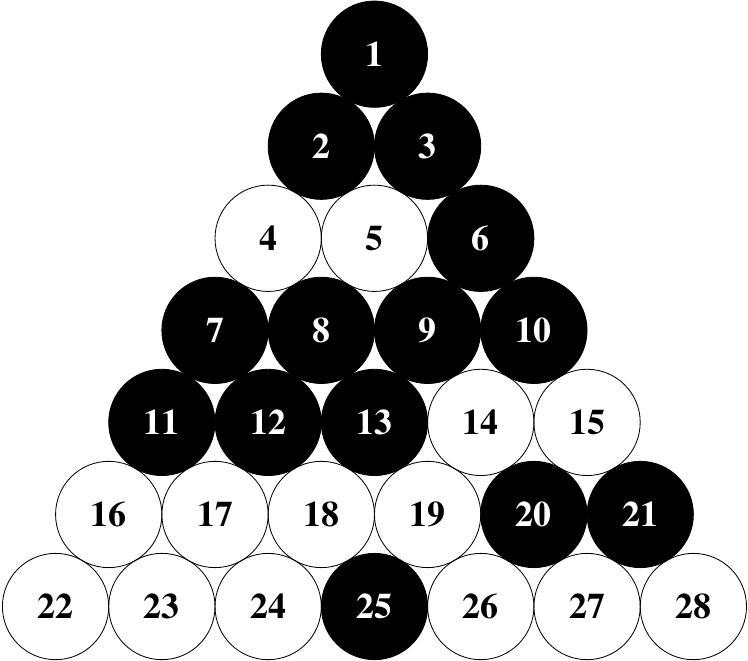}
  \caption{A configuration of the game of size $n=7$.}
  \label{fig:antes}
\end{minipage}%
\hspace{.09\textwidth}
\begin{minipage}{.45\textwidth}
  \centering
  \includegraphics[width=.85\linewidth]{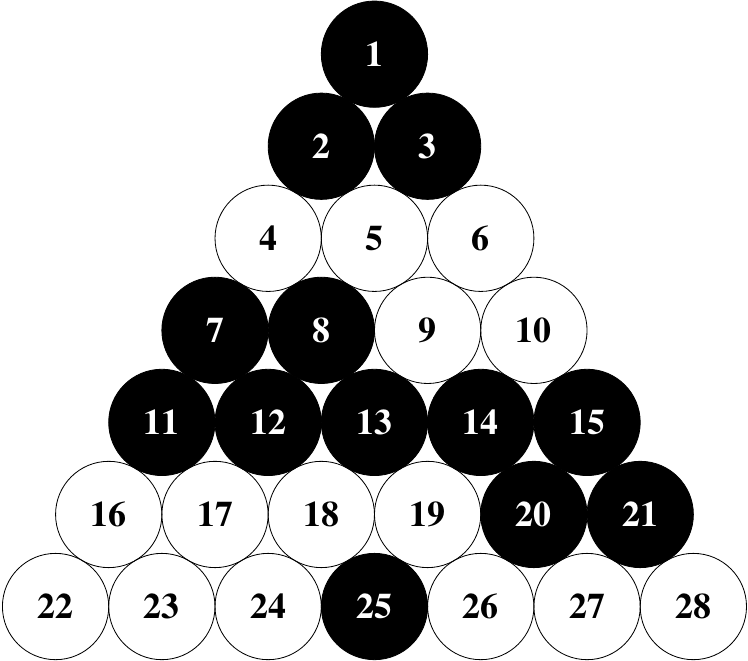}
  \caption{The configuration obtained by pressing button 10 in Figure~\ref{fig:antes}.}
  \label{fig:depois}
\end{minipage}
\end{figure}

The goal of the game is to turn off all the lights. In other words, given an initial configuration, the player must determine a sequence of buttons whose pressing results in the {all-off configuration}, that is, the one where all buttons are turned off. In the example from {Figure~\ref{fig:antes}}, one possible sequence is 1, 10, 12, 20, 22, 23, 25. An interactive version of the game, implemented in GeoGebra, can be found at~\cite{cassio:geogebra}.

This scenario raises several natural questions. For a given configuration, does a solution always exist? When it does, is it necessarily unique? And to what extent do these answers depend on the initial configuration or the value of $n$? Some of these questions will be discussed throughout the article. In Section~\ref{sec:basics}, we present basic properties of the game and an important consequence: for a given $n$, the existence of unsolvable configurations is related to the presence of \emph{null sequences}, that is, non-trivial sequences of buttons that, when pressed, do not change the initial configuration. In Section~\ref{sec:nucleo}, we establish a criterion for determining values of $n$ for which unsolvable configurations exist.

Initially, we the authors approached the problem through the tools of linear algebra, and the results of Section~\ref{sec:basics} can be systematically formalized within that context. However, we chose to present an elementary version there that does not require any prior technical knowledge. With minor adaptations, that section can be reframed in a way that is accessible even to middle school students. In Section~\ref{sec:alg.lin}, we explicitly include the linear algebra approach due to its great effectiveness in analyzing the game. For high school students with a strong algebraic foundation, this perspective can serve as an introduction to topics such as linear systems, matrices and their inverses, and, subsequently, linear algebra itself. For students who have already mastered these topics, the problem offers a natural example of vector spaces over finite fields.

\section{Preliminaries}
\label{sec:basics}

In this section, we will develop some results regarding the existence and quantity of solutions. We begin with a few simple but fundamental observations about how the buttons function.

In a game of size $n$ , there are $\beta=\frac{n(n+1)}2$ buttons, which we will identify by the numbers $1,2,\dots,\beta$ . A first important observation is that the order in which the buttons are pressed does not influence the final result. In fact, for a fixed button $j$, its final state depends only on how many buttons that affect it were pressed, and not on the order of those presses. Furthermore, since each button has only two possible states, pressing the same button an even number of times has no effect at all. We record these observations in the following property.

\begin{property}
The order in which the buttons are pressed is irrelevant, and pressing the same button twice is equivalent to not pressing it at all.
\label{prop:ordem.nao.importa}
\end{property}

Based on this property, the same subset of $\{1,2,\dots,\beta\}$ can be interpreted in distinct ways, depending on the context. These interpretations correspond to different aspects of the game.

\textbf{Interpretation as a {State} of the game.}
A subset can represent a configuration, that is, a choice for the state of each button. In this interpretation, the elements of the subset are precisely the buttons that are lit. For example, the empty set corresponds to the configuration in which all buttons are turned off. We will use the letters $C$ and $D$ to denote configurations.

\textbf{Interpretation as an Action of the game}
A subset can also represent the set of buttons we choose to press. Thanks to Property~\ref{prop:ordem.nao.importa}, any sequence of moves can be described in this way, since the order is irrelevant and repeated presses an even number of times can be discarded. When using this interpretation, we will generally employ the letters  $X$ e $Y$.

\textbf{Interpretation as the Effect of a move}
Finally, a subset can represent the set of buttons whose state was changed when moving from one configuration to another. For example, when moving from the all-off configuration $\emptyset$ to a configuration $C$, exactly the buttons in $C$
 have had their state changed. In general, when moving from a configuration $C$ to a configuration $D$, the altered buttons form the set
\[
C \triangle D = (C \cup D)\setminus(C \cap D),
\]
that is, the symmetric difference between $C$ and $D$. We will not introduce a specific notation for this interpretation.

Using this terminology, we can describe the game as follows. Given an initial configuration $C$ and a set of buttons $X$, pressing the buttons in $X$ produces a certain effect; that is, it alters exactly the buttons of some subset $V$. The goal of the game is to choose $X$ such that this effect coincides with $C$, since, in this case, the initially lit buttons are turned off and we reach the desired final configuration. In this case, we will say that $X$ is a solution for $C$.

We observe that there are $2^\beta$ possible configurations, $2^\beta$ sets of buttons, and $2^\beta$ possible effects. However, not every effect needs to be achievable by a set of buttons. The game of size $n$ admits a solution for every configuration precisely when every possible effect can be obtained from some set of buttons. In this case, a simple count shows that the solution must be unique.

\begin{property}
In a game of size $n$, if every configuration has a solution, then this solution is unique.
\end{property}

To visually represent the evolution of configurations throughout a sequence of moves, we will adopt the notation
$$ C \xrightarrow{X} D $$
where $C$ represents the initial configuration, $D$ the final configuration, and $X$ the set of buttons pressed. In particular, $X$ is a solution for the configuration $C$  when $C \xrightarrow{X}\emptyset$.

Suppose that $C \xrightarrow{X_1} C_1$ and then $C_1 \xrightarrow{X_2} C_2$. Since pressing buttons present simultaneously in $X_1$ and $X_2$ is equivalent to not pressing them, the total effect of these two steps is obtained by pressing only the buttons in $(X_1 \cup X_2) \setminus (X_1 \cap X_2) = X_1 \triangle X_2 $. Thus, $$ C \xrightarrow{X_1} C_1 \xrightarrow{X_2} C_2 \qquad\Longrightarrow\qquad C \xrightarrow{X_1 \triangle X_2} C_2. $$ In other words, the symmetric difference of sets of buttons produces the same effect as the sequence of the corresponding moves.

Let us now turn to analyzing the converse of Property~\ref{prop:nucleo.tem.k}. Suppose that a configuration $C$ admits two distinct solutions. That is, there exist sets of buttons $X_1$ and $X_2$, with $X_1 \neq X_2$, such that $C \xrightarrow{X_1} \emptyset $ e $C \xrightarrow{X_2} \emptyset$. We note that this implies that $\emptyset \xrightarrow{X_2} C$ and from this we obtain $$ C \xrightarrow{X_1} \emptyset \xrightarrow{X_2} C \qquad\Longrightarrow\qquad C \xrightarrow{X_1 \triangle X_2} C. $$ Since $X_1 \neq X_2$, it follows that $X = X_1 \triangle X_2$ is non-empty. Therefore, the effect of pressing exactly the buttons in $X$ is to change no button at all. This type of set plays a central role in the game.

\begin{definition}
A set of buttons $X \subset \{1,2,...,\beta\}$ is a \emph{null sequence} when $C \xrightarrow{X}C$ for every configuration $C$. The set of all null sequences is called the \emph{kernel} of the game.
\end{definition}

We observe that the empty set is always a null sequence, since pressing no button does not change any configuration. Furthermore, if a set $X$ satisfies $C \xrightarrow{X} C$ for some configuration $C$, then the same holds for any configuration. Thus, to verify if $X$ belongs to the kernel, it suffices to check whether $\emptyset \xrightarrow{X} \emptyset$. In other words, it is enough to check if $X$ is a solution for the configuration  $\emptyset$.

As we have seen, the existence of two distinct solutions for the same configuration implies the existence of a non-trivial null sequence. In a sense, distinct solutions for the same configuration always differ by a null sequence. Observing this allows us to describe all solutions for a configuration (that admits one). If the kernel has exactly $k$ elements, say $N_1, \dots, N_k$, and if $X$ is a solution for a configuration $C$, then
$$ X \triangle N_1,\; X \triangle N_2,\; \dots,\; X \triangle N_k $$
are precisely all the solutions for $C$. We will prove this statement below.

\begin{property}
In a game of size $n$, if the kernel has $k$ elements, then every configuration that has a solution has exactly $k$ solutions. In particular, if $k > 1$, there exist configurations with no solution.
\label{prop:nucleo.tem.k}
\end{property}

\begin{proof}
Let $N_1,\dots,N_k$ be the elements of the kernel. Let $C$ be a configuration that admits a solution, and let $X$ such that $C \xrightarrow{X} \emptyset$. Since each $N_j$ belongs to the kernel, we have $\emptyset \xrightarrow{N_j} \emptyset$ and, therefore, $$ C \xrightarrow{X} \emptyset \xrightarrow{N_j} \emptyset \qquad\Longrightarrow\qquad C \xrightarrow{X\triangle N_j} \emptyset. $$ Thus, each $X\triangle N_j$ is a solution for $C$. These solutions are distinct from each other because $$ X\triangle N_i = X\triangle N_j \;\Longrightarrow\; (X\triangle N_i)\triangle X = (X\triangle N_j)\triangle X \;\Longrightarrow\; N_i=N_j. $$

It remains to be shown that there are no other solutions. Let $Y$ be an additional solution for $C$. Then $$ \emptyset \xrightarrow{X} C \xrightarrow{Y} \emptyset \qquad\Longrightarrow\qquad \emptyset \xrightarrow{X\triangle Y} \emptyset. $$
Thus,  $X\triangle Y$ belongs to the kernel, say $X\triangle Y=N_j$, which implies $$ Y = X\triangle (X\triangle Y) = X\triangle N_j. $$
This completes the proof.
\end{proof}

To conclude this section, we determine the kernel for the case of games of size $n=2$. The case $n=1$ is trivial, as there is only one button and pressing it changes any configuration; therefore, the kernel contains only the empty set.

For $n=2$, we have $3$ buttons. We observe that pressing any one of them produces exactly the same effect, and pressing two distinct buttons is equivalent to producing no change at all. Thus, the kernel contains four elements:
$$ \emptyset,\quad \{1,2\},\quad \{1,3\},\quad \{2,3\}. $$
These combinations are represented in Figure~\ref{fig:nucleo.2}.

\begin{figure}[H]
    \centering
    \includegraphics[width=0.22\linewidth]{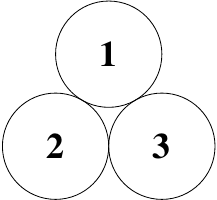}
    \hspace{2mm}
    \includegraphics[width=0.22\linewidth]{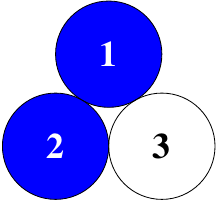}
    \hspace{2mm}
    \includegraphics[width=0.22\linewidth]{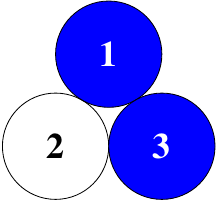}
    \hspace{2mm}
    \includegraphics[width=0.22\linewidth]{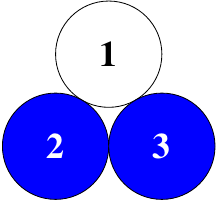}
    \caption{The kernel of a game of size $n = 2$.}
    \label{fig:nucleo.2}
\end{figure}

\section{Kernels}
\label{sec:nucleo}

In the previous section, we defined the kernel. An element of the kernel is a \emph{null sequence}; that is, a set of buttons that, when pressed, do not change the initial configuration. We proved that a game of size $n$ has a unique solution for any configuration when its kernel is trivial. That is, the only element of the kernel is the empty set.

In this section, we will investigate for which values of $n$ the kernel is non-trivial. We will represent the elements of the kernel through figures. For example, Figure \ref{fig:nucleo.exemplo} below represents an element $N$ of the kernel of a game of size  $6$. In this figure, the buttons colored in blue are the elements of $N$. In other words, starting from a configuration $C$, if we press the blue buttons, the final configuration will be $C$.

\begin{figure}[H]
    \centering
    \includegraphics[width=0.25\linewidth]{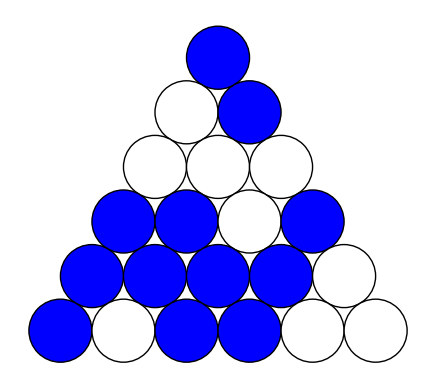}
    \caption{A representation of an element of the kernel of a game of size $6$.}
    \label{fig:nucleo.exemplo}
\end{figure}

To begin, we will verify that the number of elements in the kernel is always a power of 2.

\begin{property}
In a game of size $n$, the kernel has $2^{\ell}$ elements, for same integer $\ell$ such that $0 \leq \ell \leq \dfrac{n(n+1)}{2}$.
\end{property}

\begin{proof}
Given a set of buttons $X$, let $C$ be the set of buttons altered when we press $X$. Considering $C$ as a game configuration, this means that $X$ is a solution for $C$. Thus, every set of buttons is a solution to some configuration.

On the other hand, let $k$ be the number of elements in the kernel. By Property~\ref{prop:nucleo.tem.k}, every solvable configuration admits exactly $k$ solutions. If $m$ denotes the number of solvable configurations, then the total number of solutions is $m \cdot k$. Since each set of buttons is the solution to exactly one configuration, we conclude that the total number of sets of buttons is $m \cdot k$.

The total number of sets of buttons is $2^{\tfrac{n(n+1)}{2}}$, hence $m \cdot k = 2^{\tfrac{n(n+1)}{2}}$.  Therefore, $k$ must be a power of $2$, that is, $k = 2^{\ell}$ for some integer $\ell$ whith $0 \leq \ell \leq \dfrac{n(n+1)}{2}$.
\end{proof}

\begin{observacao}
In the previous property, the interval $0 \le \ell \le \dfrac{n(n+1)}{2}$ can be refined. In fact, it is not difficult to verify that necessarily $\ell \le n$. A simple argument consists of observing that two distinct elements of the kernel must differ in at least one button of the last row of the triangular array, which contains exactly $n$ buttons.
\end{observacao}

{We will say that the kernel has {dimension} $\ell$ when it has exactly $2^{\ell}$ elements.} Table~\ref{tab:nucleos} presents the values of this dimension for $1 \le n \le 80$. For example, for $n=5$, we observe that $\ell=2$, which means that the kernel contains $2^{2}=4$ elements. These elements are represented in Figure~\ref{fig:nucleo.5}

\begin{table}[H]
    \centering
\begin{tabular}{||c|c||c|c||c|c||c|c||c|c||c|c||c|c||c|c||} \hline 
$n$  & $\ell$ & $n$  & $\ell$ & $n$  & $\ell$ & $n$  & $\ell$ & $n$  & $\ell$ & $n$  & $\ell$ & $n$  & $\ell$ & $n$  & $\ell$  \\ \hline 
$1$  & $0$ & $11$ & $0$ & $21$ & $0$ & $31$ & $0$ & $41$ & $0$ & $51$ & $0$ & $61$ & $20$ & $71$ & $6$ \\
$2$  & $2$ & $12$ & $4$ & $22$ & $4$ & $32$ & $0$ & $42$ & $2$ & $52$ & $0$ & $62$ & $42$ & $72$ & $0$ \\
$3$  & $0$ & $13$ & $4$ & $23$ & $0$ & $33$ & $2$ & $43$ & $4$ & $53$ & $0$ & $63$ & $0$  & $73$ & $4$ \\
$4$  & $0$ & $14$ & $10$ & $24$ & $0$ & $34$ & $2$ & $44$ & $0$ & $54$ & $20$ & $64$ & $0$ & $74$ & $2$ \\
$5$  & $2$ & $15$ & $0$ & $25$ & $0$ & $35$ & $0$ & $45$ & $0$ & $55$ & $20$ & $65$ & $0$ & $75$ & $0$ \\
$6$  & $4$ & $16$ & $0$ & $26$ & $10$ & $36$ & $4$ & $46$ & $10$ & $56$ & $18$ & $66$ & $4$ & $76$ & $10$ \\
$7$  & $0$ & $17$ & $0$ & $27$ & $0$ & $37$ & $0$ & $47$ & $2$ & $57$ & $0$ & $67$ & $0$ & $77$ & $0$ \\
$8$  & $0$ & $18$ & $2$ & $28$ & $8$ & $38$ & $4$ & $48$ & $0$ & $58$ & $18$ & $68$ & $4$ & $78$ & $10$ \\
$9$  & $0$ & $19$ & $8$ & $29$ & $10$ & $39$ & $0$ & $49$ & $0$ & $59$ & $0$ & $69$ & $0$ & $79$ & $0$ \\
$10$ & $2$ & $20$ & $0$ & $30$ & $20$ & $40$ & $16$ & $50$ & $2$ & $60$ & $20$ & $70$ & $4$ & $80$ & $0$ \\

 \hline 

\end{tabular}
    \caption{Kernel dimension for games of size 1 to 80.}
    \label{tab:nucleos}
\end{table}

\begin{figure}[h]
    \centering
    \includegraphics[width=0.24\linewidth]{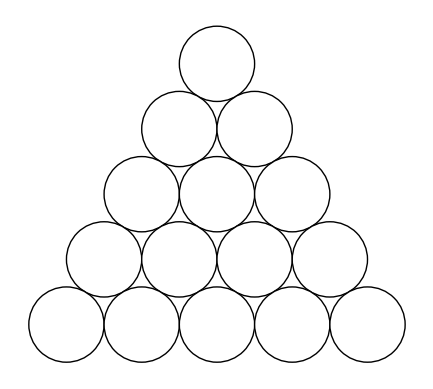}
    \includegraphics[width=0.24\linewidth]{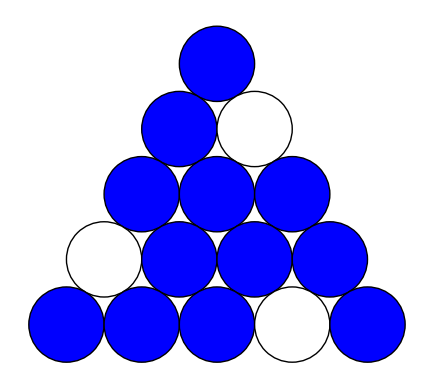}
    \includegraphics[width=0.24\linewidth]{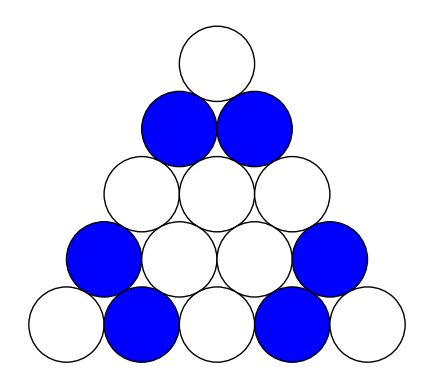}
    \includegraphics[width=0.24\linewidth]{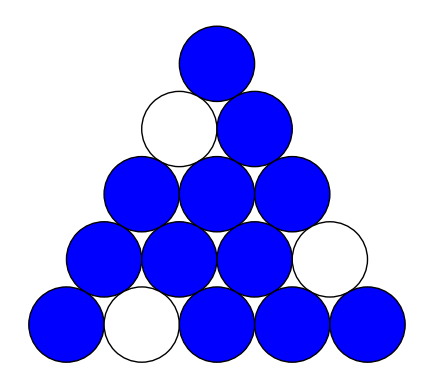}
    \caption{Elements of the kernel of a game of size $n=5$.}
    \label{fig:nucleo.5}
\end{figure}

For $n=22$, we have $\ell=4$. This means that the kernel has $2^{4}=16$ elements, which are represented in Figure~\ref{fig:nucleo.22}. This case is interesting because all elements exhibit clear symmetries. For $n=19$, the kernel has $2^{8}=256$ elements. In Figure~\ref{fig:nucleo.19}, we show some of them. We observe that some present evident symmetry, while others do not exhibit any apparent pattern.

\begin{figure}[H]
    \centering
    \includegraphics[width=0.12\linewidth]{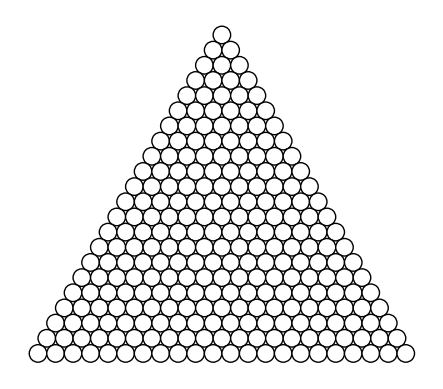}
    \includegraphics[width=0.12\linewidth]{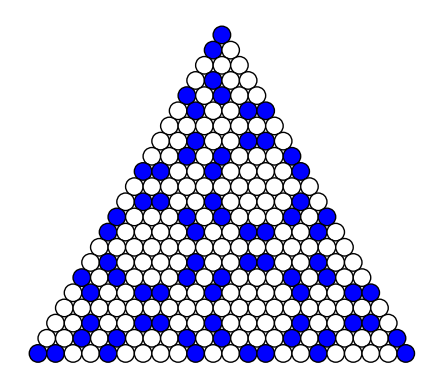}
    \includegraphics[width=0.12\linewidth]{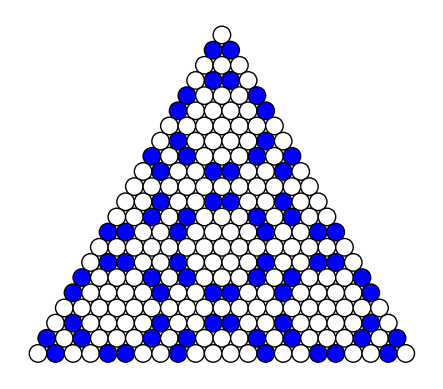}
    \includegraphics[width=0.12\linewidth]{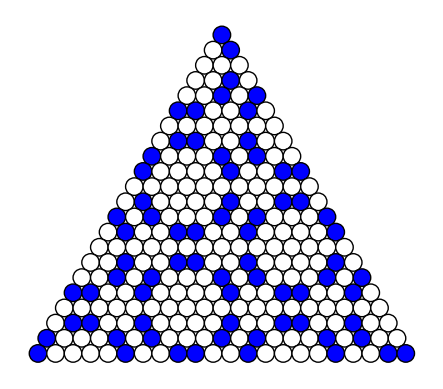}
    \includegraphics[width=0.12\linewidth]{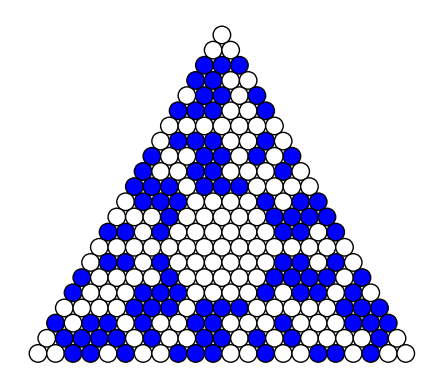}
    \includegraphics[width=0.12\linewidth]{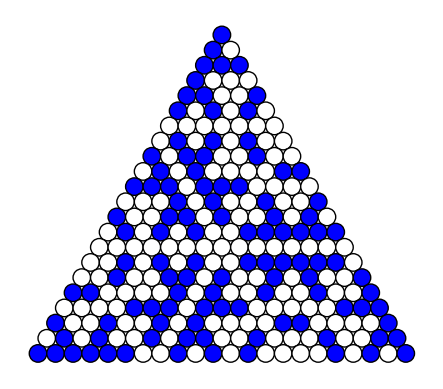}
    \includegraphics[width=0.12\linewidth]{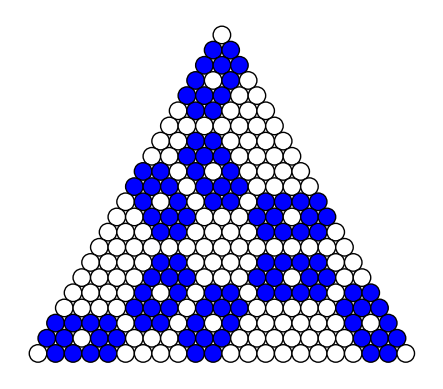}
    \includegraphics[width=0.12\linewidth]{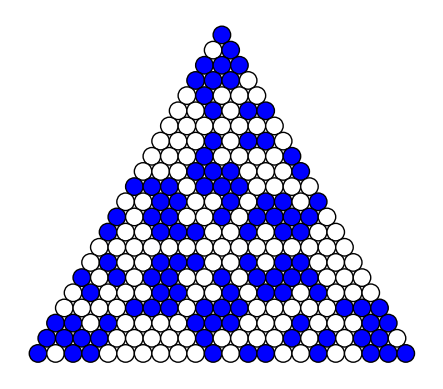}
    \includegraphics[width=0.12\linewidth]{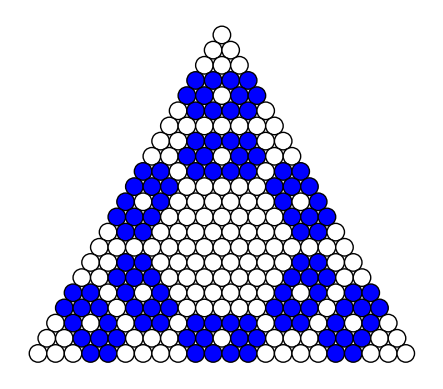}
    \includegraphics[width=0.12\linewidth]{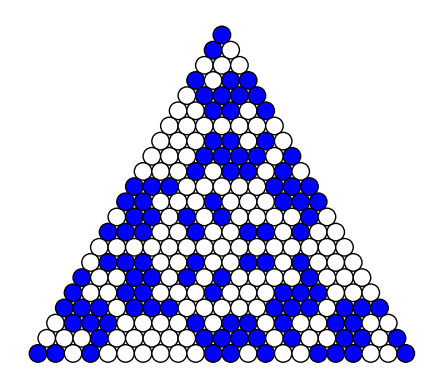}
    \includegraphics[width=0.12\linewidth]{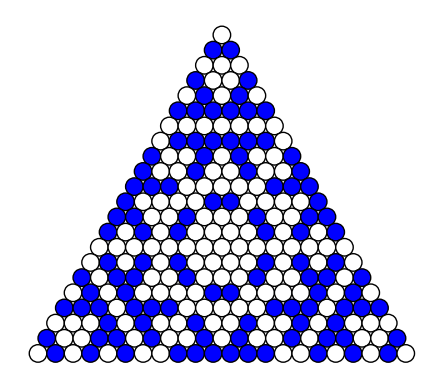}
    \includegraphics[width=0.12\linewidth]{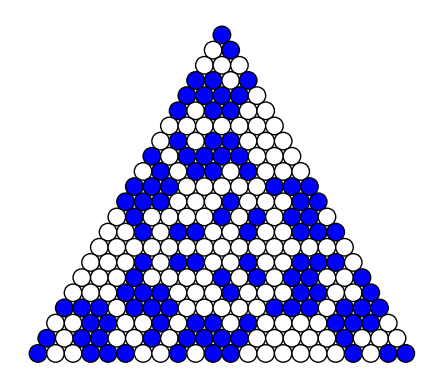}
    \includegraphics[width=0.12\linewidth]{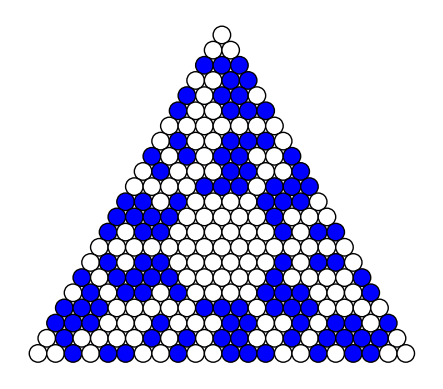}
    \includegraphics[width=0.12\linewidth]{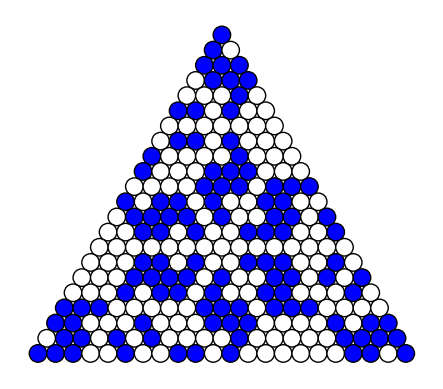}
    \includegraphics[width=0.12\linewidth]{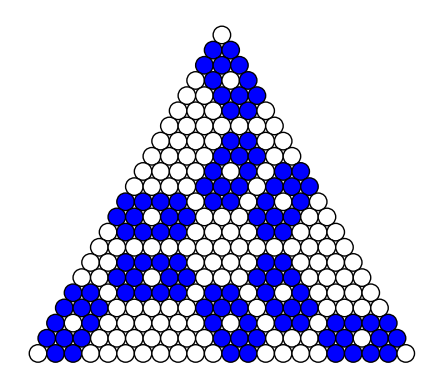}
    \includegraphics[width=0.12\linewidth]{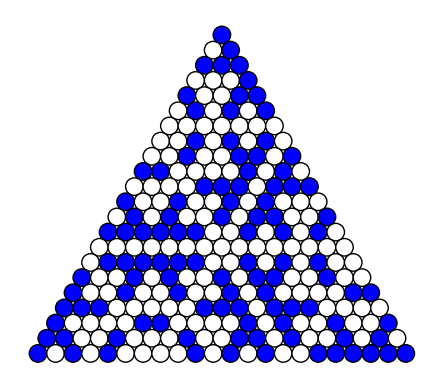}

    \caption{Elements of the kernel of a game of size $n=22$.}
    \label{fig:nucleo.22}
\end{figure}

\begin{figure}[H]
    \centering
    \includegraphics[width=0.12\linewidth]{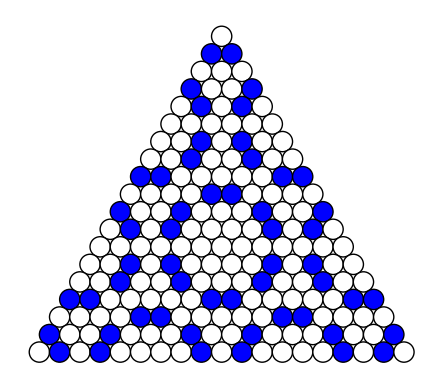}
    \includegraphics[width=0.12\linewidth]{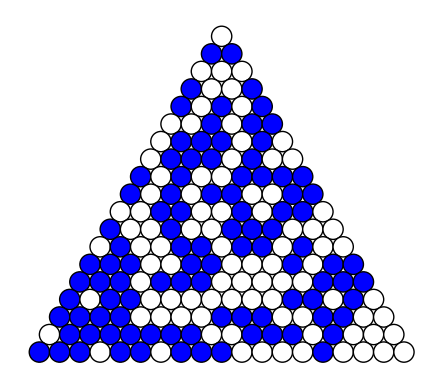}
    \includegraphics[width=0.12\linewidth]{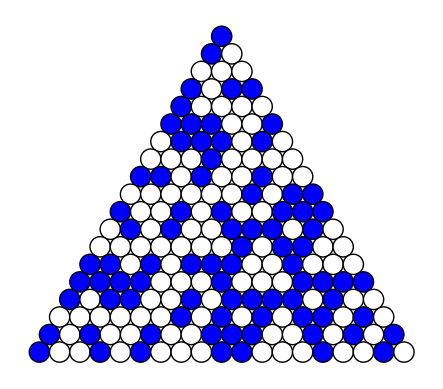}
    \includegraphics[width=0.12\linewidth]{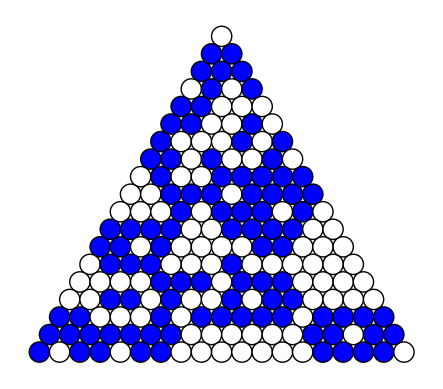}
    \includegraphics[width=0.12\linewidth]{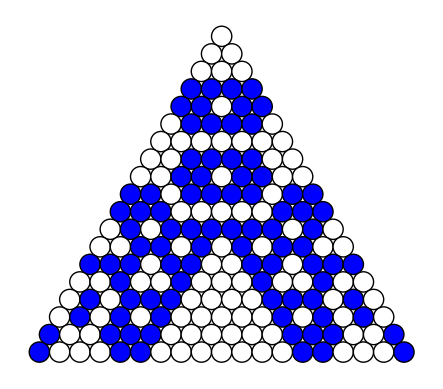}
    \includegraphics[width=0.12\linewidth]{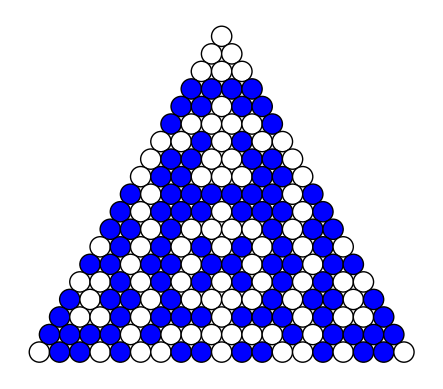}
    \includegraphics[width=0.12\linewidth]{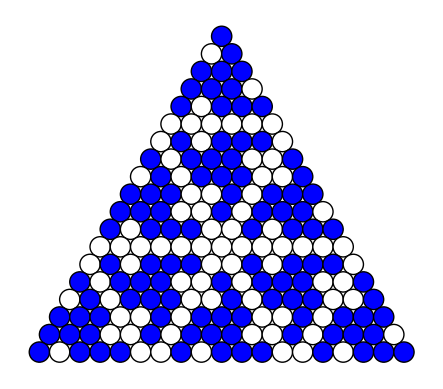}
    \includegraphics[width=0.12\linewidth]{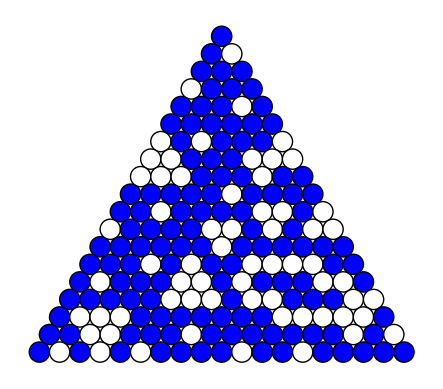}
    \caption{Some elements of the kernel of a game of size $n=19$.}
    \label{fig:nucleo.19}
\end{figure}

Even with few examples, it is possible to notice certain recurring designs in the symmetric figures. Some of these patterns appear in games of different sizes. Figures~\ref{fig:padrao.1} and~\ref{fig:padrao.2} illustrate some of these cases.

\begin{figure}[H]
    \centering
    \includegraphics[width=0.19\linewidth]{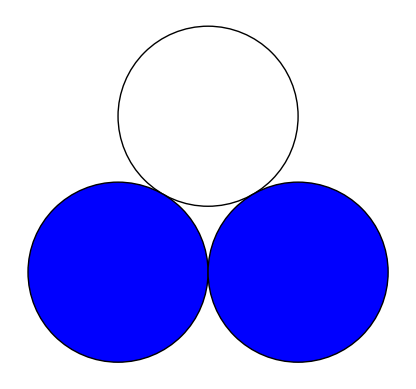}
    \includegraphics[width=0.19\linewidth]{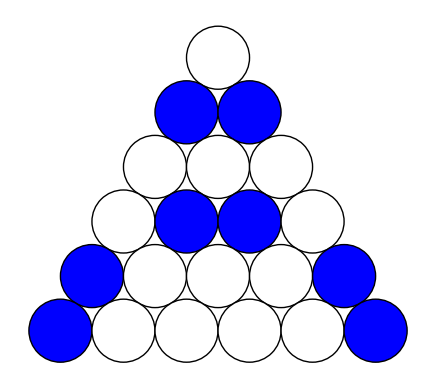}
    \includegraphics[width=0.19\linewidth]{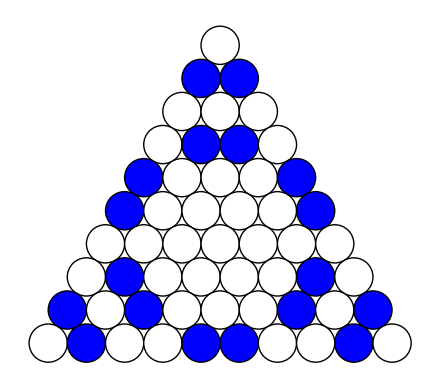}
    \includegraphics[width=0.19\linewidth]{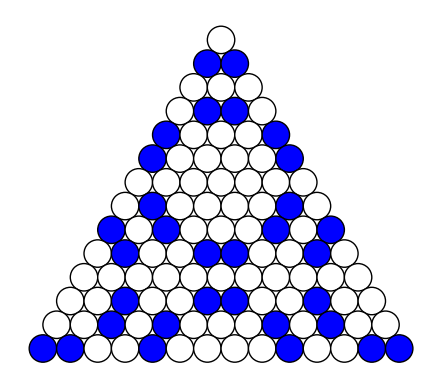}
    \includegraphics[width=0.19\linewidth]{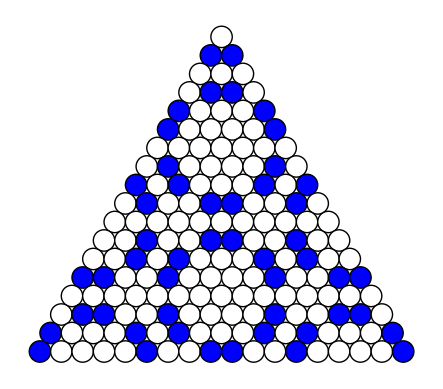}
    \includegraphics[width=0.19\linewidth]{resources/triangulo_22_3.png}
    \includegraphics[width=0.19\linewidth]{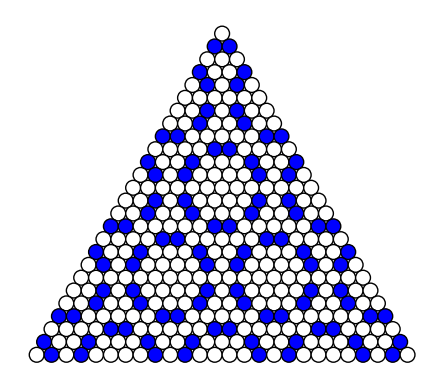}
    \includegraphics[width=0.19\linewidth]{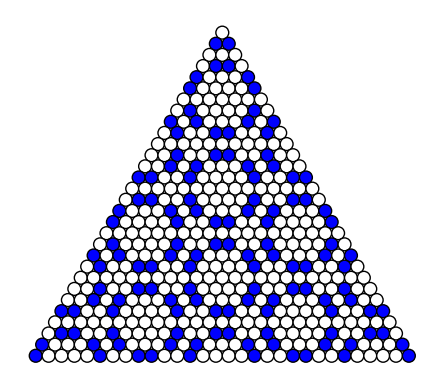}
    \includegraphics[width=0.19\linewidth]{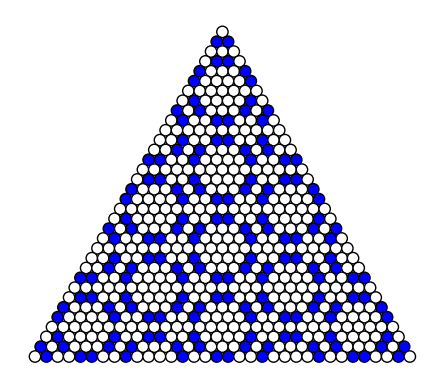}
    \includegraphics[width=0.19\linewidth]{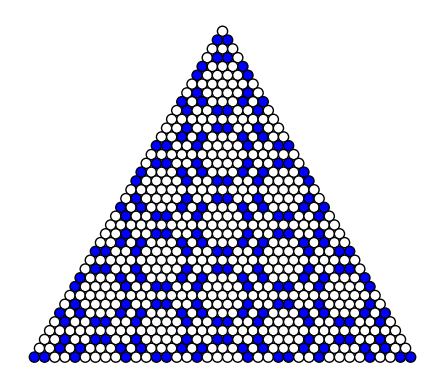}

    \caption{Some patterns.}
    \label{fig:padrao.1}
\end{figure}
\begin{figure}[H]
    \centering
    
    \includegraphics[width=0.19\linewidth]{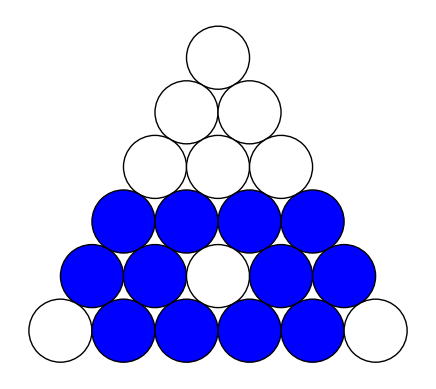}
    \includegraphics[width=0.19\linewidth]{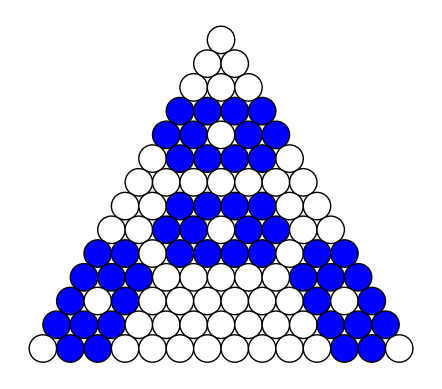}
    \includegraphics[width=0.19\linewidth]{resources/triangulo_22_9.png}
    \includegraphics[width=0.19\linewidth]{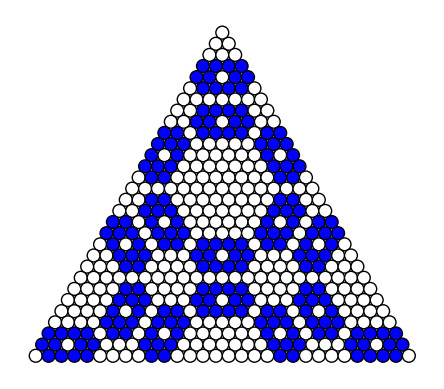}
    \includegraphics[width=0.19\linewidth]{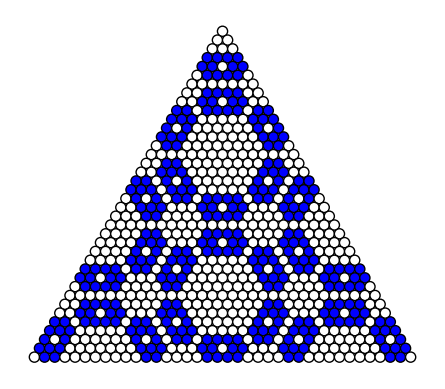}
    
    \includegraphics[width=0.19\linewidth]{resources/triangulo_5_2.png}
    \includegraphics[width=0.19\linewidth]{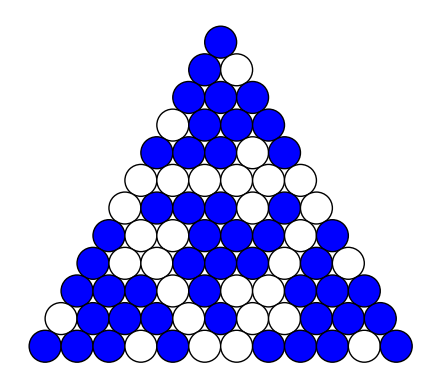}
    \includegraphics[width=0.19\linewidth]{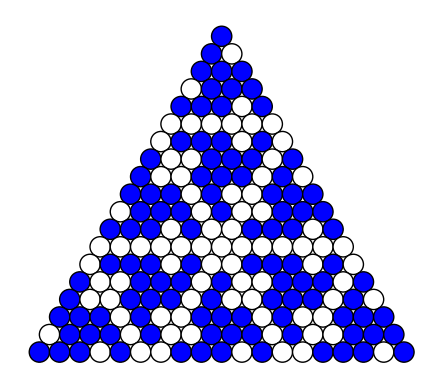}
    \includegraphics[width=0.19\linewidth]{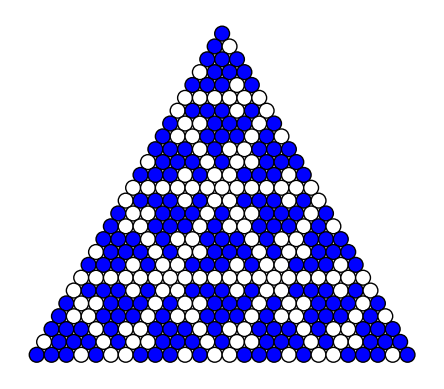}
    \includegraphics[width=0.19\linewidth]{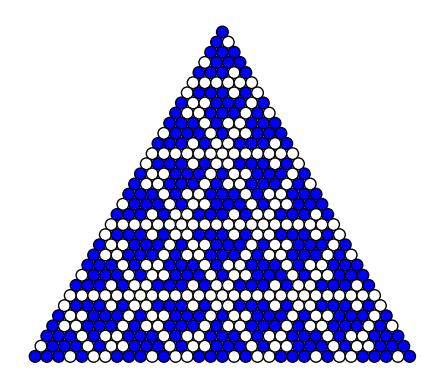}

    \includegraphics[width=0.19\linewidth]{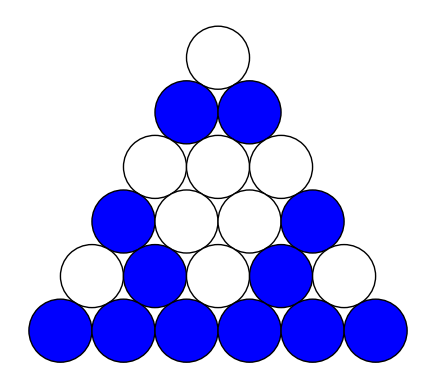}
    \includegraphics[width=0.19\linewidth]{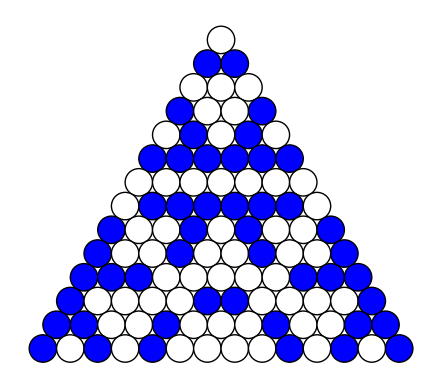}
    \includegraphics[width=0.19\linewidth]{resources/triangulo_22_11.png}
    \includegraphics[width=0.19\linewidth]{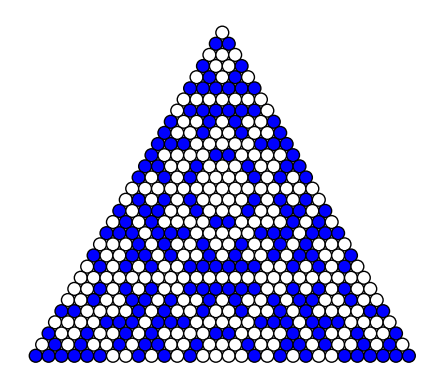}
    \includegraphics[width=0.19\linewidth]{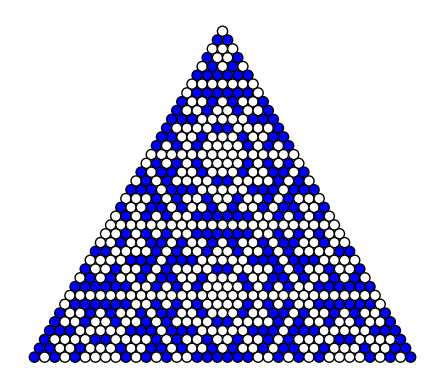}
    
    \includegraphics[width=0.19\linewidth]{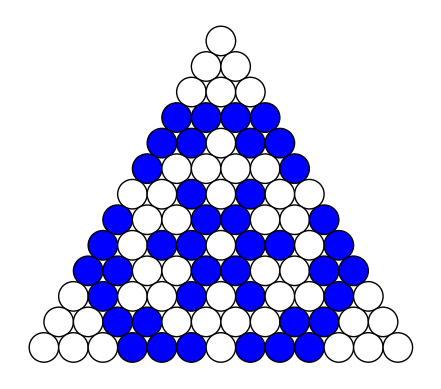}
    \includegraphics[width=0.19\linewidth]{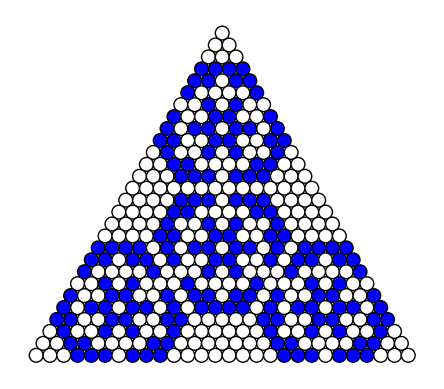}
    \includegraphics[width=0.19\linewidth]{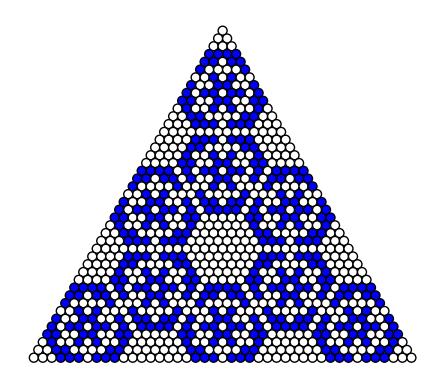}
    \includegraphics[width=0.19\linewidth]{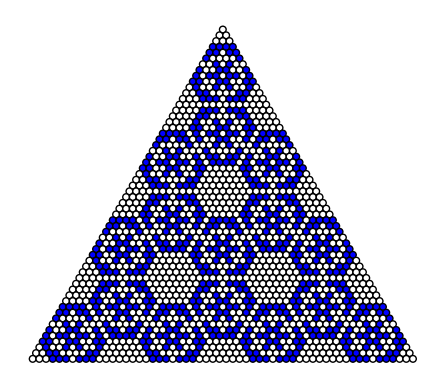}
    \includegraphics[width=0.19\linewidth]{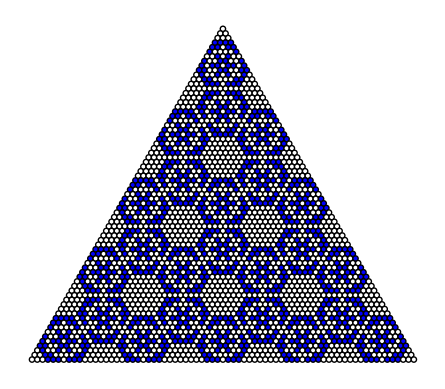}
    \caption{Further patterns.}
    \label{fig:padrao.2}
\end{figure}

We wish to understand the nature of this phenomenon—that is, to investigate how these patterns arise and how they propagate. The observation of the sequence presented in Figure~\ref{fig:padrao.1} points toward some structural regularity. In this figure, the games have sizes $2$, $6$, $10$, $14$, $18$, $22$, $26$, $30$, $34$ and $38$. These values suggest that games of size $4j+2$ present kernel elements exhibiting similar patterns. This is, in fact, the case.

Still in Figure~\ref{fig:padrao.1}, the mechanism of pattern propagation is not always evident, largely because most of the buttons do not appear colored. In contrast, in Figure~\ref{fig:padrao.2}, especially in the second and third rows, this mechanism becomes quite clear. The visible pattern is obtained by repeating kernel elements from smaller games and their symmetries, separated by a row of unused buttons. Figure~\ref{fig:propagacao.padrao} explicitly illustrates this process. The arrows and colors indicate which rotations and reflections are applied to the smaller patterns to compose the element observed in the larger games.

\begin{figure}[H]
    \centering
    \includegraphics[width=0.5\linewidth]{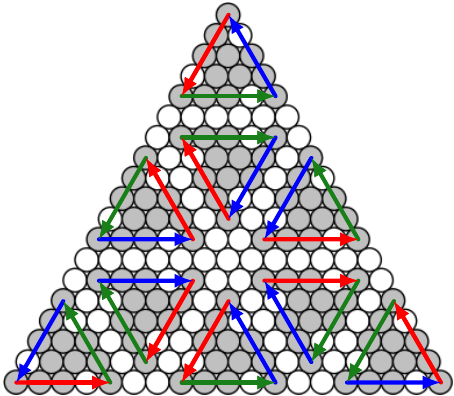}
    \caption{A mechanism of pattern propagation.}
    \label{fig:propagacao.padrao}
\end{figure}

These observations motivate the following result.

\begin{theorem}
If a game of size $n$ has a non-trivial kernel, then the same holds for all games of size $n+(n+2)j$, for every integer $j \geq 1$.
\end{theorem}

\begin{proof}
Let $T$ be an element of the kernel of a game of size $n$. By the symmetry of the triangular array, any reflection of $T$ along its heights and any rotation of $T$ also belongs to the kernel. For each integer $k\ge 1$, we will construct an element of the kernel of a game of size $n+(n+2)j$ from copies of $T$ and its symmetries.

The construction is illustrated in Figure~\ref{fig:propagacao.padrao}. In it, each triangular region marked by arrows corresponds to a block of size $n$ formed by $T$ or by one of its symmetries. These blocks are arranged so as to fill a larger triangulation and are separated by lines of width $1$ consisting entirely of unused buttons. Let $X$ be the set of buttons obtained by merging all these blocks, we will show that $X$ belongs to the kernel of the game of size $n+(n+2)j$.

 Consider an arbitrary button $b$. If $b$ belongs to the interior of one of the triangular blocks, then the only buttons in $X$
that affect $b$ are those belonging to the block itself. Since this block is a copy of $T$ or one of its symmetries, it is an element of the kernel of the game of size $n$. Thus, pressing the buttons in this block leaves the state of $b$ unchanged.

It remains to consider the buttons belonging to the dividing lines between the blocks. None of these buttons are contained in $X$. Furthermore, each of them has an even number of neighbors in $X$. This occurs because such buttons are on the boundary between two blocks that are congruent by reflection, which ensures that the neighbors present in $X$ appear in symmetric pairs around $b$. Since pressing two identical neighbors does not change the state of $b$, the final result for these buttons also coincides with their initial state.

We conclude that, for any button $b$, pressing the buttons in $X$ does not change the state of $b$. Therefore, $X$ belongs to the kernel of the game of size $n+(n+2)j$, which completes the proof.
\end{proof}

To conclude this section, we observe that there are many other mechanisms of pattern propagation. For games of size $2^{j}-3$, with $j \ge 3$, it is possible to construct a kernel element whose design resembles the fractal known as the Sierpinski Triangle. This case appears in Figure~\ref{fig:padrao.3}.

\begin{figure}[H]
    \centering
    \includegraphics[width=0.24\linewidth]{resources/triangulo_5_3.png}
    \includegraphics[width=0.24\linewidth]{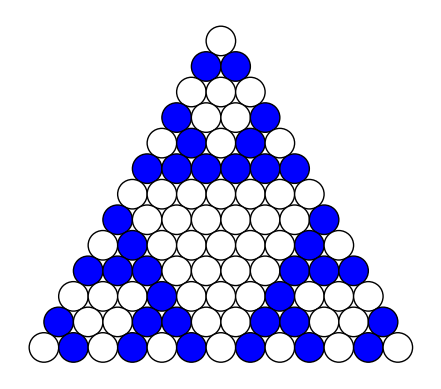}
    \includegraphics[width=0.24\linewidth]{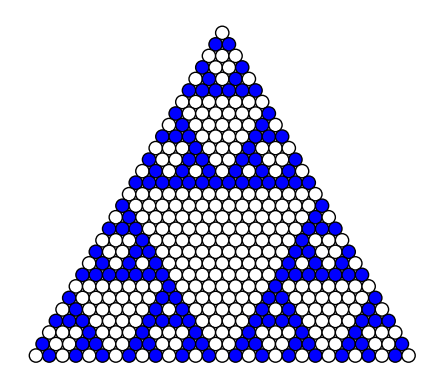}
    \includegraphics[width=0.24\linewidth]{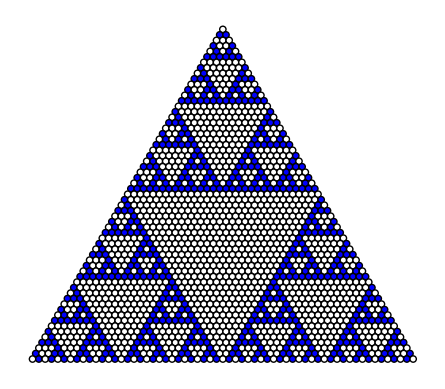}

    \caption{Kernel elements resembling the Sierpinski Triangle.}
    \label{fig:padrao.3}
\end{figure} 

Various other shapes can also be obtained. Among the symmetric figures, most of the visually striking images present patterns that resemble a kaleidoscope, such as the first two in Figure~\ref{fig:padrao.4}. There are also configurations with a less usual appearance, such as the third image in the same figure, whose arrangement suggests the shape of a spider. These examples show the surprising variety of patterns and illustrate the visual richness of the kernel elements.

\begin{figure}[H]
    \centering
    \includegraphics[width=0.32\linewidth]{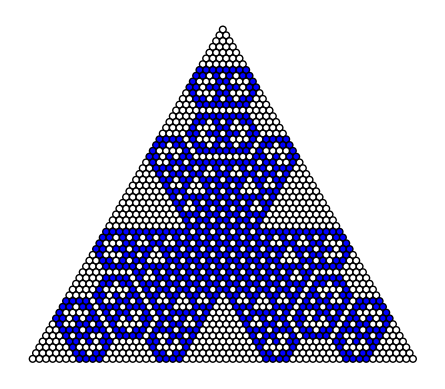}
    \includegraphics[width=0.32\linewidth]{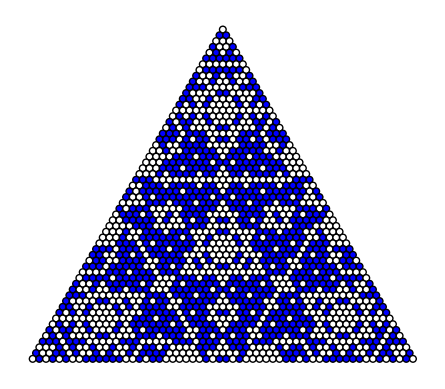}
    \includegraphics[width=0.32\linewidth]{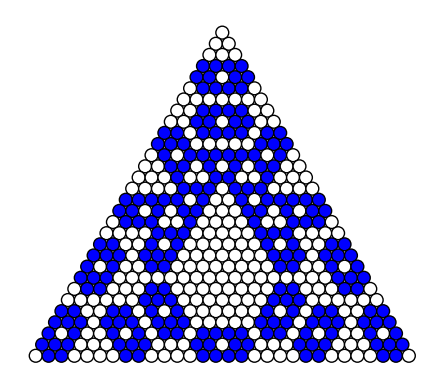}

    \caption{Kaleidoscope and Spider.}
    \label{fig:padrao.4}
\end{figure}

\section{Linear Algebra}
\label{sec:alg.lin}
In this section, we will describe the game using linear systems. In a game of size $n$, there are $\beta=\tfrac{n(n+1)}{2}$ buttons, and each of them can be either on or off. For reference purposes, we will number the buttons from $1$ to $\beta$.

Our goal is as follows. Given an initial configuration, that is, a choice of 'on' or 'off' states for each of the $\beta$  buttons, we wish to determine which buttons must be pressed so that, in the end, all of them are off. From the results in Section~\ref{sec:basics}, we know that the order of the presses does not affect the outcome. Therefore, what truly matters is only the set of buttons that will be pressed.

Let us now fix a button $j$. To understand how its final state will be determined, we analyze the effects that each button $i$ can exert on it. There are four possibilities, which we represent in the flowchart below. In each case, we indicate whether $i$ affects the state of $j$ or not, and whether button $i$ is pressed or not."

 \centerline{
\includegraphics[width = 150mm]{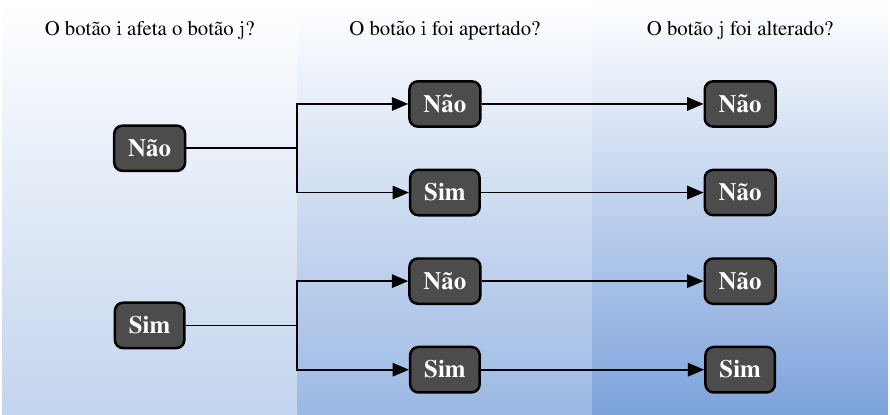}
}

Of the four situations described in the flowchart, we observe that the first three do not produce any change in the state of button $j$. Only the fourth situation, in which button $i$ affects $j$ and is indeed pressed, modifies its state. To translate these possibilities into algebraic language, we introduce the following notation.

We associate with each solution a collection of numbers

$$
x_{1}, x_{2},\dots, x_{\beta},
$$
where each $x_{i}$ can take only the values $0$ or $1$. We interpret $x_{i}=0$ as 'button $i$ was not pressed' and $x_{i}=1$ as 'button $i$ was pressed'. We also need to indicate which buttons influence the state of others. To do this, we define the coefficients $a_{ij}$ as follows:
$$
a_{ij}=
\begin{cases}
1, & \text{if pressing button $i$ alters the state of button $j$},\\[4pt]
0, & \text{otherwise}.
\end{cases}
$$

According to the flowchart, the state of button $j$ is altered by $i$ only when  $x_{i}=1$ and $a_{ij}=1$. In particular,

$$
a_{ij}x_{i}=
\begin{cases}
0, & \text{if pressing button $i$ does not alter the state of button $j$},\\[4pt]
1, & \text{if button $i$ alters the state of $j$ and was pressed}.
\end{cases}
$$

Thus, the sum
$$
a_{1j}x_{1}+a_{2j}x_{2}+\cdots+a_{\beta j}x_{\beta}
$$
counts exactly how many times the state of button $j$ was modified throughout the solution. Since each modification toggles the state from 'on' to 'off' or vice-versa, the value of this sum determines the final state of $j$: if the sum is even, the final state coincides with the initial one; if it is odd, the final state differs from the initial one.

We observe that, to determine the final state of $j$, only the parity of the sum matters
$$
a_{1j}x_{1}+a_{2j}x_{2}+\cdots+a_{\beta j}x_{\beta}.
$$

For this reason, from this point on we will work modulo $2$. In particular, every sum result will be $0$ (the even case) or $1$ (the odd case), such that $1+1=0$, $1+0=0+1=1$ and $0+0=0$. This arithmetic corresponds exactly to the behavior of the game: pressing a button is equivalent to adding $1$ to the current state, switching 'on' to 'off' and 'off' to 'on'.

Let $c_{j}$ be the initial state of button $j$, with $c_{j}=0$ for 'off' and $c_{j}=1$ for 'on'. The final state of $j$ is then given by
$$
a_{1j}x_{1}+a_{2j}x_{2}+\cdots+a_{\beta j}x_{\beta}+c_{j}.
$$
The vector $(x_{1},\dots,x_{\beta})$ solves the problem when all final states are equal to zero. Therefore,
$$
a_{1j}x_{1}+a_{2j}x_{2}+\cdots+a_{\beta j}x_{\beta}+b_{j}=0
\qquad \text{for all } j.
$$

By gathering these $\beta$ equations, we obtain the linear system
$$\left\{\begin{array}{c}
    a_{11}x_1 +   \cdots + a_{\beta 1}x_\beta + c_1 = 0 \\ 
    a_{12}x_1 +  \cdots + a_{\beta 2}x_\beta + c_2 =0\\
    \vdots \\
    a_{1\beta}x_1 +  \cdots + a_{\beta \beta }x_\beta + c_\beta = 0
\end{array}\right. \Longleftrightarrow \begin{bmatrix}
    a_{11} &  \cdots & a_{1\beta} \\
    a_{21} &  \cdots & a_{2\beta} \\
    \vdots &  \ddots & \vdots \\
    a_{\beta1} & \cdots & a_{\beta\beta}
\end{bmatrix} \cdot \begin{bmatrix}
    x_1 \\ \vdots \\ x_\beta
\end{bmatrix} + \begin{bmatrix}
    c_1  \\ \vdots \\ c_\beta
\end{bmatrix} = \begin{bmatrix}
    0 \\ \vdots \\ 0
\end{bmatrix}$$

In matrix notation, we write
\[
A=(a_{ij}), \quad 
x=(x_{1},\dots,x_{\beta}) \quad
\text{e}
\quad 
c=(c_{1},\dots,c_{\beta});
\] 
Thus, the system takes the compact form $Ax+c = 0$, which is equivalent (we are working modulo 2!) to $Ax = c$.

The fundamental question becomes: for which vectors $c$ does the system admit a solution? From linear algebra, we know that if the matrix $A$ is invertible, then there is a unique solution for every $c$. We thus obtain the following result.

\begin{theorem}
A game of size $n$ admits a unique solution for any initial configuration if, and only if, the matrix $A=(a_{ij})$ is invertible.
\end{theorem}

Before proceeding, let us relate the language of this section to the language used in previous sections. If $C$ is an initial configuration and $X$ is a set of buttons that solves it, we write $C \xrightarrow{X} \emptyset$. To the configuration $C$, we associate the vector  $c=(c_{1},\dots,c_{\beta})$, and to the set of buttons $X$, we associate the vector $x=(x_{1},\dots,x_{\beta})$. As seen above, the condition for $X$ to turn off all the lights is $Ax+c=0$. Thus, the operation $C \xrightarrow{X} \emptyset$ is recorded by the equation $Ax + c = 0$. Both languages describe the same phenomenon.

More generally, an operation of the type $C \xrightarrow{X} D$ corresponds to the equality
$$
Ax + c = d,
$$
where $Ax$ represents the set of buttons whose state is altered by pressing $X$. By adding it to $c$, we obtain the final state $d$ associated with the configuration $D$.

Let us now consider a chain of operations,
$C \xrightarrow{X} E \xrightarrow{Y} D$. We can also write
$$
\begin{cases}
Ax + c = e,\\[2mm]
Ay + e = d.
\end{cases}
$$
Adding these two equalities modulo $2$, we obtain
$A(x+y) + c = d$.
Therefore, pressing the buttons in $X$ followed by those in $Y$ produces the same effect as pressing the set of buttons described by $x+y$. By arithmetic modulo $2$, the vector $x+y$ corresponds to the symmetric difference $X \triangle Y$. Thus,
$$
C \xrightarrow{X\triangle Y} D,
$$
as we already knew from the previous section. In algebraic terms, the map that associates each set of buttons $X$ with the vector $x$ is a group homomorphism (the group operations here are the symmetric difference of sets and the addition of vectors).

Let us look at an example. Figure~\ref{fig:3-antes} shows a state of a game of size $3$ along with a possible numbering of its six buttons. Suppose we press buttons $3$ and $4$. This alters the states of buttons $1$, $3$, $4$ and $6$, producing the configuration shown in Figure~\ref{fig:3-depois}.

Let us describe this same operation with matrices and vectors. For the adopted numbering, the matrix $A$ is

\begin{equation}
\label{eq.matriz.3}
A = \begin{bmatrix}
    1&1&1&0&0&0\\
    1&1&1&1&1&0\\
    1&1&1&0&1&1\\
    0&1&0&1&1&0\\
    0&1&1&1&1&0\\
    0&0&1&0&1&1
\end{bmatrix}.
\end{equation}

Pressing buttons $3$ and $4$ corresponds to the vector $x = (0,0,1,1,0,0)$. By calculating $Ax$, we obtain

$$
Ax =
\begin{bmatrix}
    1&1&1&0&0&0\\
    1&1&1&1&1&0\\
    1&1&1&0&1&1\\
    0&1&0&1&1&0\\
    0&1&1&1&1&0\\
    0&0&1&0&1&1
\end{bmatrix}
\!\cdot\!
\begin{bmatrix}
    0\\0\\1\\1\\0\\0
\end{bmatrix}
=
\begin{bmatrix}
    1\\0\\1\\1\\0\\1
\end{bmatrix}.
$$

This vector indicates exactly which buttons will have their states swapped: indices $1$, $3$, $4$ e $6$, as expected.

If the initial configuration is represented by the vector
$$
c = (0,1,0,0,0,1),
$$
then the final configuration is given by
$$
Ax + c = (1,1,1,1,0,0),
$$
which coincides with the situation illustrated in Figure~\ref{fig:3-depois}.

\begin{figure}[H]
\centering
\begin{minipage}{.45\textwidth}
  \centering
  \includegraphics[width=.65\linewidth]{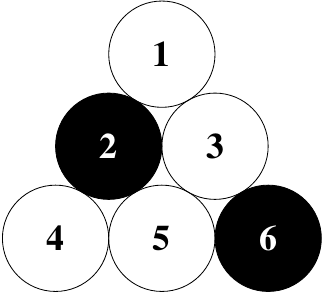}
  \caption{A configuration of a game of size $n=3$.}
  \label{fig:3-antes}
\end{minipage}%
\hspace{.09\textwidth}
\begin{minipage}{.45\textwidth}
  \centering
  \includegraphics[width=.65\linewidth]{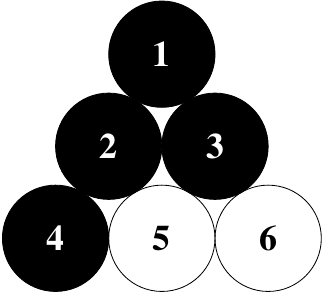}
  \caption{The configuration obtained by pressing buttons $3$ and $4$ in Figure~\ref{fig:3-antes}.}
  \label{fig:3-depois}
\end{minipage}
\end{figure}

Let us now analyze the kernel elements discussed in the previous sections. We recall that a set of buttons $X$ belongs to the kernel when, for any configuration $B$, it holds that
$$
C \xrightarrow{X} C,
$$
that is, pressing the buttons in $X$ produces no change at all. In the language of linear systems, this means
$$
Ax + c = c,
$$
or rather, $Ax=0$. Therefore, the elements of the kernel of the matrix $A$ correspond exactly to the elements of the kernel of the game (which justifies the terminology)

From linear algebra, we know that $A$ is invertible if, and only if, the homogeneous system $Ax=0$ has only the trivial solution. Furthermore, for any $c$, the system $Ax=c$ has the same number of solutions as the homogeneous system. Since we are working modulo $2$, the kernel is a vector subspace over the field $\mathbb Z_2$, and therefore has cardinality $2^\ell$ for some $\ell$. This is precisely Property~\ref{prop:nucleo.tem.k} presented in Section~\ref{sec:basics}. The algorithm employed to generate the figures in the previous section consists of calculating the kernel of the matrix $A$ corresponding to the game.

We now proceed to the explicit study of the kernel of $A$ for some values of $n$ (we will denote $A$ by $A_n$ for differentiation). 
To standardize, we will always consider the numbering of the buttons from top to bottom and from left to right, as in Figures~\ref{fig:antes} e \ref{fig:3-antes}. 

For $n=1$, there is only one button, and every configuration has a unique solution. The matrix is
$$
A_1 = \begin{bmatrix} 1 \end{bmatrix},
$$
which is invertible, and therefore
$
\textrm{Núcleo}(A_1) = \left\{ \begin{bmatrix} 0 \end{bmatrix} \right\}
$.

For $n=2$, there are three buttons, and we obtain
$$
A_2 = 
\begin{bmatrix}
    1 & 1 & 1\\
    1 & 1 & 1\\
    1 & 1 & 1
\end{bmatrix}.
$$
The kernel is
$$
\textrm{Núcleo}(A_2)=
\left\{
\begin{bmatrix} 0\\0\\0 \end{bmatrix},
\begin{bmatrix} 1\\1\\0 \end{bmatrix},
\begin{bmatrix} 0\\1\\1 \end{bmatrix},
\begin{bmatrix} 1\\0\\1 \end{bmatrix}
\right\},
$$
and its elements are illustrated in Figure~\ref{fig:nucleo.exemplo}. 

Since there are $2^3 = 8$ possible configurations and the kernel has $4$ elements, it follows from Property~\ref{prop:nucleo.tem.k} that each solvable configuration admits exactly $4$ solutions. Consequently, only two configurations are solvable, shown in Figure~\ref{fig:2-soluveis}. The others, represented in Figure~\ref{fig:2-insoluveis}, have no solution.

\begin{figure}[H]
\centering
\begin{minipage}{.25\textwidth}
  \centering
  \includegraphics[width=25mm]{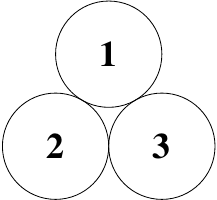}\vspace{1mm}
  \includegraphics[width=25mm]{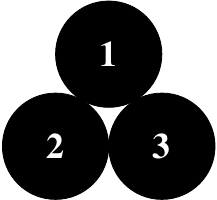}
  \caption{Solvable configurations of the game of size $n=2$.}
  \label{fig:2-soluveis}
\end{minipage}%
\hspace{.05\textwidth}
\begin{minipage}{.60\textwidth}
  \centering
  \includegraphics[width=25mm]{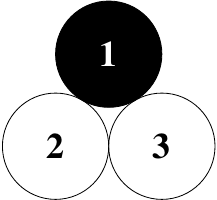}\hspace{3mm}
  \includegraphics[width=25mm]{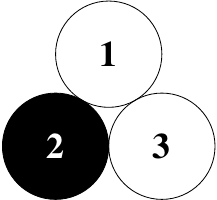}\hspace{3mm}
  \includegraphics[width=25mm]{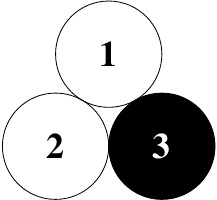}\vspace{1mm}
  \includegraphics[width=25mm]{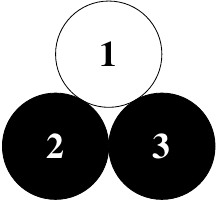}\hspace{3mm}
  \includegraphics[width=25mm]{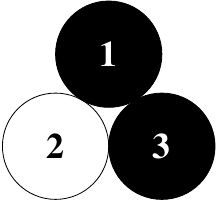}\hspace{3mm}
  \includegraphics[width=25mm]{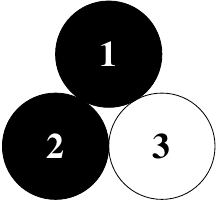}
  \caption{Unsolvable configurations of the game of size $n=2$.}
  \label{fig:2-insoluveis}
\end{minipage}
\end{figure}

For $n=3$, the matrix $A_3$ is the one presented in equation~\eqref{eq.matriz.3}. In this case, $A_3$ is invertible. The same occurs for $n=4$; that is, $A_4$ is also invertible. For $n=5$, the kernel of $A_5$ has exactly four elements, represented in Figure~\ref{fig:nucleo.5} from the previous section. For larger values of $n$, the calculations quickly become unwieldy. However, upon explicitly inspecting the matrices $A_1, \dots, A_5$, a pattern becomes visible: with the numbering we have adopted, each matrix $A_n$ is embedded as a block of $A_{n+1}$. This regularity is not accidental and can be verified directly.

To illustrate this phenomenon, we present below the matrix $A_6$, in which we highlight the blocks corresponding to $A_1, A_2, \dots, A_5$. We also highlight all non-zero entries.

% fonte: https://www.tablesgenerator.com/#

$$ A_6 = \left[
\begin{array}{lllllllllllllllllllll}
\multicolumn{1}{l|}{{\color[HTML]{3166FF} \textbf{1}}} & {\color[HTML]{3166FF} \textbf{1}} & \multicolumn{1}{l|}{{\color[HTML]{3166FF} \textbf{1}}} & 0 & 0 & \multicolumn{1}{l|}{0} & 0 & 0 & 0 & \multicolumn{1}{l|}{0} & 0 & 0 & 0 & 0 & \multicolumn{1}{l|}{0} & 0 & 0 & 0 & 0 & 0 & 0 \\ \cline{1-1}
{\color[HTML]{3166FF} \textbf{1}} & {\color[HTML]{3166FF} \textbf{1}} & \multicolumn{1}{l|}{{\color[HTML]{3166FF} \textbf{1}}} & {\color[HTML]{3166FF} \textbf{1}} & {\color[HTML]{3166FF} \textbf{1}} & \multicolumn{1}{l|}{0} & 0 & 0 & 0 & \multicolumn{1}{l|}{0} & 0 & 0 & 0 & 0 & \multicolumn{1}{l|}{0} & 0 & 0 & 0 & 0 & 0 & 0 \\
{\color[HTML]{3166FF} \textbf{1}} & {\color[HTML]{3166FF} \textbf{1}} & \multicolumn{1}{l|}{{\color[HTML]{3166FF} \textbf{1}}} & 0 & {\color[HTML]{3166FF} \textbf{1}} & \multicolumn{1}{l|}{{\color[HTML]{3166FF} \textbf{1}}} & 0 & 0 & 0 & \multicolumn{1}{l|}{0} & 0 & 0 & 0 & 0 & \multicolumn{1}{l|}{0} & 0 & 0 & 0 & 0 & 0 & 0 \\ \cline{1-3}
0 & {\color[HTML]{3166FF} \textbf{1}} & 0 & {\color[HTML]{3166FF} \textbf{1}} & {\color[HTML]{3166FF} \textbf{1}} & \multicolumn{1}{l|}{0} & {\color[HTML]{3166FF} \textbf{1}} & {\color[HTML]{3166FF} \textbf{1}} & 0 & \multicolumn{1}{l|}{0} & 0 & 0 & 0 & 0 & \multicolumn{1}{l|}{0} & 0 & 0 & 0 & 0 & 0 & 0 \\
0 & {\color[HTML]{3166FF} \textbf{1}} & {\color[HTML]{3166FF} \textbf{1}} & {\color[HTML]{3166FF} \textbf{1}} & {\color[HTML]{3166FF} \textbf{1}} & \multicolumn{1}{l|}{{\color[HTML]{3166FF} \textbf{1}}} & 0 & {\color[HTML]{3166FF} \textbf{1}} & {\color[HTML]{3166FF} \textbf{1}} & \multicolumn{1}{l|}{0} & 0 & 0 & 0 & 0 & \multicolumn{1}{l|}{0} & 0 & 0 & 0 & 0 & 0 & 0 \\
0 & 0 & {\color[HTML]{3166FF} \textbf{1}} & 0 & {\color[HTML]{3166FF} \textbf{1}} & \multicolumn{1}{l|}{{\color[HTML]{3166FF} \textbf{1}}} & 0 & 0 & {\color[HTML]{3166FF} \textbf{1}} & \multicolumn{1}{l|}{{\color[HTML]{3166FF} \textbf{1}}} & 0 & 0 & 0 & 0 & \multicolumn{1}{l|}{0} & 0 & 0 & 0 & 0 & 0 & 0 \\ \cline{1-6}
0 & 0 & 0 & {\color[HTML]{3166FF} \textbf{1}} & 0 & 0 & {\color[HTML]{3166FF} \textbf{1}} & {\color[HTML]{3166FF} \textbf{1}} & 0 & \multicolumn{1}{l|}{0} & {\color[HTML]{3166FF} \textbf{1}} & {\color[HTML]{3166FF} \textbf{1}} & 0 & 0 & \multicolumn{1}{l|}{0} & 0 & 0 & 0 & 0 & 0 & 0 \\
0 & 0 & 0 & {\color[HTML]{3166FF} \textbf{1}} & {\color[HTML]{3166FF} \textbf{1}} & 0 & {\color[HTML]{3166FF} \textbf{1}} & {\color[HTML]{3166FF} \textbf{1}} & {\color[HTML]{3166FF} \textbf{1}} & \multicolumn{1}{l|}{0} & 0 & {\color[HTML]{3166FF} \textbf{1}} & {\color[HTML]{3166FF} \textbf{1}} & 0 & \multicolumn{1}{l|}{0} & 0 & 0 & 0 & 0 & 0 & 0 \\
0 & 0 & 0 & 0 & {\color[HTML]{3166FF} \textbf{1}} & {\color[HTML]{3166FF} \textbf{1}} & 0 & {\color[HTML]{3166FF} \textbf{1}} & {\color[HTML]{3166FF} \textbf{1}} & \multicolumn{1}{l|}{{\color[HTML]{3166FF} \textbf{1}}} & 0 & 0 & {\color[HTML]{3166FF} \textbf{1}} & {\color[HTML]{3166FF} \textbf{1}} & \multicolumn{1}{l|}{0} & 0 & 0 & 0 & 0 & 0 & 0 \\
0 & 0 & 0 & 0 & 0 & {\color[HTML]{3166FF} \textbf{1}} & 0 & 0 & {\color[HTML]{3166FF} \textbf{1}} & \multicolumn{1}{l|}{{\color[HTML]{3166FF} \textbf{1}}} & 0 & 0 & 0 & {\color[HTML]{3166FF} \textbf{1}} & \multicolumn{1}{l|}{{\color[HTML]{3166FF} \textbf{1}}} & 0 & 0 & 0 & 0 & 0 & 0 \\ \cline{1-10}
0 & 0 & 0 & 0 & 0 & 0 & {\color[HTML]{3166FF} \textbf{1}} & 0 & 0 & 0 & {\color[HTML]{3166FF} \textbf{1}} & {\color[HTML]{3166FF} \textbf{1}} & 0 & 0 & \multicolumn{1}{l|}{0} & {\color[HTML]{3166FF} \textbf{1}} & {\color[HTML]{3166FF} \textbf{1}} & 0 & 0 & 0 & 0 \\
0 & 0 & 0 & 0 & 0 & 0 & {\color[HTML]{3166FF} \textbf{1}} & {\color[HTML]{3166FF} \textbf{1}} & 0 & 0 & {\color[HTML]{3166FF} \textbf{1}} & {\color[HTML]{3166FF} \textbf{1}} & {\color[HTML]{3166FF} \textbf{1}} & 0 & \multicolumn{1}{l|}{0} & 0 & {\color[HTML]{3166FF} \textbf{1}} & {\color[HTML]{3166FF} \textbf{1}} & 0 & 0 & 0 \\
0 & 0 & 0 & 0 & 0 & 0 & 0 & {\color[HTML]{3166FF} \textbf{1}} & {\color[HTML]{3166FF} \textbf{1}} & 0 & 0 & {\color[HTML]{3166FF} \textbf{1}} & {\color[HTML]{3166FF} \textbf{1}} & {\color[HTML]{3166FF} \textbf{1}} & \multicolumn{1}{l|}{0} & 0 & 0 & {\color[HTML]{3166FF} \textbf{1}} & {\color[HTML]{3166FF} \textbf{1}} & 0 & 0 \\
0 & 0 & 0 & 0 & 0 & 0 & 0 & 0 & {\color[HTML]{3166FF} \textbf{1}} & {\color[HTML]{3166FF} \textbf{1}} & 0 & 0 & {\color[HTML]{3166FF} \textbf{1}} & {\color[HTML]{3166FF} \textbf{1}} & \multicolumn{1}{l|}{{\color[HTML]{3166FF} \textbf{1}}} & 0 & 0 & 0 & {\color[HTML]{3166FF} \textbf{1}} & {\color[HTML]{3166FF} \textbf{1}} & 0 \\
0 & 0 & 0 & 0 & 0 & 0 & 0 & 0 & 0 & {\color[HTML]{3166FF} \textbf{1}} & 0 & 0 & 0 & {\color[HTML]{3166FF} \textbf{1}} & \multicolumn{1}{l|}{{\color[HTML]{3166FF} \textbf{1}}} & 0 & 0 & 0 & 0 & {\color[HTML]{3166FF} \textbf{1}} & {\color[HTML]{3166FF} \textbf{1}} \\ \cline{1-15}
0 & 0 & 0 & 0 & 0 & 0 & 0 & 0 & 0 & 0 & {\color[HTML]{3166FF} \textbf{1}} & 0 & 0 & 0 & 0 & {\color[HTML]{3166FF} \textbf{1}} & {\color[HTML]{3166FF} \textbf{1}} & 0 & 0 & 0 & 0 \\
0 & 0 & 0 & 0 & 0 & 0 & 0 & 0 & 0 & 0 & {\color[HTML]{3166FF} \textbf{1}} & {\color[HTML]{3166FF} \textbf{1}} & 0 & 0 & 0 & {\color[HTML]{3166FF} \textbf{1}} & {\color[HTML]{3166FF} \textbf{1}} & {\color[HTML]{3166FF} \textbf{1}} & 0 & 0 & 0 \\
0 & 0 & 0 & 0 & 0 & 0 & 0 & 0 & 0 & 0 & 0 & {\color[HTML]{3166FF} \textbf{1}} & {\color[HTML]{3166FF} \textbf{1}} & 0 & 0 & 0 & {\color[HTML]{3166FF} \textbf{1}} & {\color[HTML]{3166FF} \textbf{1}} & {\color[HTML]{3166FF} \textbf{1}} & 0 & 0 \\
0 & 0 & 0 & 0 & 0 & 0 & 0 & 0 & 0 & 0 & 0 & 0 & {\color[HTML]{3166FF} \textbf{1}} & {\color[HTML]{3166FF} \textbf{1}} & 0 & 0 & 0 & {\color[HTML]{3166FF} \textbf{1}} & {\color[HTML]{3166FF} \textbf{1}} & {\color[HTML]{3166FF} \textbf{1}} & 0 \\
0 & 0 & 0 & 0 & 0 & 0 & 0 & 0 & 0 & 0 & 0 & 0 & 0 & {\color[HTML]{3166FF} \textbf{1}} & {\color[HTML]{3166FF} \textbf{1}} & 0 & 0 & 0 & {\color[HTML]{3166FF} \textbf{1}} & {\color[HTML]{3166FF} \textbf{1}} & {\color[HTML]{3166FF} \textbf{1}} \\
0 & 0 & 0 & 0 & 0 & 0 & 0 & 0 & 0 & 0 & 0 & 0 & 0 & 0 & {\color[HTML]{3166FF} \textbf{1}} & 0 & 0 & 0 & 0 & {\color[HTML]{3166FF} \textbf{1}} & {\color[HTML]{3166FF} \textbf{1}}
\end{array} \right] $$

We conclude this section with a result whose proof combines linear algebra tools with a classic combinatorics problem. A similar formulation for the game in a square arrangement is presented in~\cite{jaap:site}. It was from this source that we became aware of the result, whose simplicity and elegance surprised us and motivated its inclusion here in the corresponding version for the triangular arrangement.

Before stating the theorem, we clarify what we mean by a covering of the board by $1\times 1$ and $2\times 1$
 pieces. This refers to a partition of the set $\{1,2,\dots,\beta\}$, that is, the set of the game's buttons, into subsets of one or two elements. When a subset has two elements, these must correspond to neighboring buttons. For example, in the case $n=4$, Figure~\ref{fig:recobrimento} illustrates a possible covering, which corresponds to the partition $\{1,2\},\{3\},\{4\},\{5,9\},\{6,10\},\{7,8\}$ of $\{1,2,\dots,10\}$.

\begin{figure}[H]
    \centering
    \includegraphics[width=0.25\linewidth]{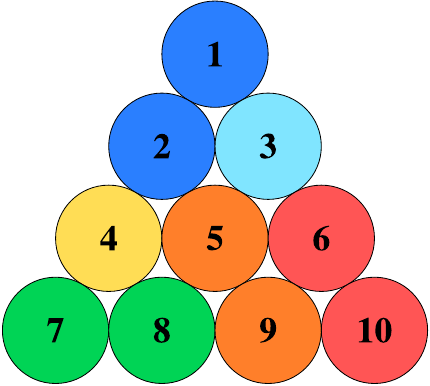}
    \caption{A way to cover the board of a game of size $n=4$ with pieces $1\times 1$ e $2 \times 1$.}
    \label{fig:recobrimento}
\end{figure}

\begin{theorem}
A game of size $n$ has a solution for every configuration if, and only if, the number of coverings of the triangular board by $1\times 1$ and $2\times 2$ pieces is odd.
\end{theorem}

\begin{proof}
We know that a game of size $n$ admits a solution for every configuration if, and only if, the matrix $A$ is invertible. Since we are working modulo $2$, this occurs exactly when $\det A=1$.

  By Leibniz's formula,
    $$
        \det A
        = \sum_{\sigma\in S_\beta}
        a_{1\sigma(1)}\, a_{2\sigma(2)} \cdots a_{\beta\sigma(\beta)},
    $$
where $S_\beta$ is the group of permutations of $\{1,2,\dots,\beta\}$. Modulo $2$, the sign of the permutation is not relevant. Each term of the sum is either $0$ or $1$, and it is equal to $1$ exactly when
    $$
        a_{1\sigma(1)} = a_{2\sigma(2)} = \cdots = a_{\beta\sigma(\beta)} = 1.
    $$

We recall that $a_{ij}=1$ only when $i=j$ or when buttons $i$ and $j$ are neighbors. Thus, let $V$ be the set of permutations $\sigma$ such that, for every $j$, the index $\sigma(j)$ coincides with $j$ or with a neighbor of $j$. For these permutations, the corresponding term in the sum of the determinant is equal to $1$; for the others, it is equal to $0$. Therefore,
    $$
        \det A = \sum_{\sigma\in V} 1.
    $$

We now observe that if $\sigma\in V$, then $\sigma^{-1}\in V$, since neighborhood is a symmetric relation. Furthermore, if $\sigma\neq\sigma^{-1}$, the two corresponding terms in the sum cancel each other out modulo $2$. Therefore, only involutive permutations effectively contribute to the final value of the determinant. Denoting by $W$ the set of permutations $\sigma\in V$ such that $\sigma = \sigma^{-1}$, we have:
    $$
        \det A = \sum_{\sigma\in W} 1.
    $$

Each permutation $\sigma\in W$ is an involution without multiple fixed points; therefore, its cycle decomposition consists solely of cycles of length $1$ and $2$. Furthermore, since $\sigma\in V$, whenever a transposed cycle $(i\, \sigma(i))$ occurs, the indices $i$ and $\sigma(i)$ correspond to neighboring buttons. Thus, each $\sigma\in W$ determines a covering of the board by $1\times 1$ pieces (the cycles of length $1$) and $2\times 1$ pieces (the cycles of length $2$). Conversely, each covering defines an involutive permutation in $W$.
   
Therefore, the number of elements in $W$ coincides with the number of coverings of the board. Since
    $$
        \det A = 1
        \quad\Longleftrightarrow\quad
        |W| \text{ is odd},
    $$
   it follows that the game has a solution for every configuration exactly when the number of coverings is odd, as desired..
\end{proof}

%\section{Conclusões}
%{\color{blue} Vale a pena por algo aqui? Problemas abertos?} \\
%- encontrar um critério pra quando tem solução. \\
%- encontrar padrões na matriz parece ser difícil \\
%- generalização pra grafos. o problema tem efeito global. Mexendo localmente no grafo altera globalmente e só Deus sabe o que acontece.

% agradecimentos
\section*{Acknowledgements}

%{\att É certo agradecer a sedu e fapes aqui?}
We would like to thank the entire team involved in the organization of OCMAT, especially professors Fábio Júlio Valentim, Florêncio Guimarães, Luzia Casati, Moacir Rosado, and Valmecir Bayer, members of the academic team, for the discussions, encouragement, and contributions throughout the entire process. We also thank FAPES and SEDU-ES for the institutional support in organizing OCMAT 2025.

%%%%%%%%%%%%%%%%%%%%%%%%%%%%%%%%%%%%%%%%%%%
% Referências Bibliográficas
%------------------------------------------

% Manter parâmetro do ambiente thebibliography para calcular
% largura do rótulo do item (usualmente "99") como vazio.

\end{document}